\documentclass[12pt]{article}
\usepackage{latexsym,amssymb,amsmath,amsthm,enumerate,geometry,float,cite}
\newtheorem{theorem}{Theorem}

\newtheorem{claim}{Claim}
\usepackage{lineno}
\usepackage{setspace}

\begin{document}
%\linenumbers
\onehalfspace

\title{Majority Edge-Colorings of Graphs}
\author{Felix Bock$^1$\and 
Rafa\l~Kalinowski$^2$\and
Johannes Pardey$^1$\and 
Monika Pil\'{s}niak$^2$\and 
Dieter Rautenbach$^1$\and 
Mariusz Wo\'{z}niak$^2$}
\date{}

\maketitle
\vspace{-10mm}
\begin{center}
{\small 
$^1$ Institute of Optimization and Operations Research, Ulm University, Ulm, Germany\\
\texttt{$\{$felix.bock,johannes.pardey,dieter.rautenbach$\}$@uni-ulm.de}\\[3mm]
$^2$ AGH University, Department of Discrete Mathematics, 30-059 Krakow, Poland\\
\texttt{$\{$kalinows,pilsniak,mwozniak$\}$@agh.edu.pl}}
\end{center}

\begin{abstract}
We propose the notion of a majority $k$-edge-coloring of a graph $G$,
which is an edge-coloring of $G$ with $k$ colors such that, for every vertex $u$ of $G$,
at most half the edges of $G$ incident with $u$ have the same color.
We show the best possible results 
that every graph of minimum degree at least $2$ has a majority $4$-edge-coloring,
and that every graph of minimum degree at least $4$ has a majority $3$-edge-coloring.
Furthermore, we discuss a natural variation of majority edge-colorings 
and some related open problems.\\[3mm]
\end{abstract}

\bigskip

\bigskip

\section{Introduction}

Motivated by similar notions considered for vertex-colorings,
we propose and study {\it majority edge-colorings} of graphs:
For a (finite, simple, and undirected) graph $G$, 
an edge-coloring $c:E(G)\to [k]$ 
is a {\it majority $k$-edge-coloring} 
if, for every vertex $u$ of $G$ 
and every color $\alpha$ in $[k]$, 
at most half the edges incident with $u$ 
have the color $\alpha$.
%The {\it majority chromatic index $\chi_m'(G)$} of $G$
%is the smallest integer $k$ such that $G$ 
%has a majority $k$-edge-coloring.

Before we present our results, 
we discuss some related research.
Lov\'{a}sz \cite{lo} showed that every graph $G$ has a $2$-vertex-coloring
such that, for every vertex $u$ of $G$, 
at most half the neighbors of $u$ have the same color as $u$.
For infinite graphs, this leads to the {\it Unfriendly Partition Conjecture} \cite{ah}.
Kreutzer, Oum, Seymour, van der Zypen, and Wood \cite{kr} 
showed that every digraph $D$ has a $4$-vertex-coloring such that, for every vertex $u$ of $D$, 
at most half the out-neighbors of $u$ have the same color as $u$,
and they conjecture that $3$ colors suffice.
Anholcer, Bosek, and Grytczuk \cite{anbogr} studied a choosability version for digraphs.
It follows from a result of Wood \cite{wo}
that every digraph $D$ has a $4$-arc-coloring
such that, for every vertex $u$ of $D$,
at most half the arcs leaving $u$ have the same color.
Further related research 
concerns {\it defective} or {\it frugal} edge-colorings \cite{abauhu,amepva,hislst},
where maximum degree conditions are imposed on the subgraphs 
formed by edges having the same color.

Our first result is that $2$ colors almost suffice for a majority edge-coloring.

\begin{theorem}\label{theorem-1}
Let $G$ be a connected graph.
\begin{enumerate}[(i)]
\item If $G$ has an even number of edges or $G$ contains vertices of odd degree, 
then $G$ has a $2$-edge-coloring 
such that, for every vertex $u$ of $G$, 
at most $\left\lceil\frac{d_G(u)}{2}\right\rceil$ of the edges 
incident with $u$ have the same color.
\item If $G$ has an odd number of edges, all vertices of $G$ have even degree,
and $u_G$ is any vertex of $G$, 
then $G$ has a $2$-edge-coloring 
such that, for every vertex $u$ of $G$ distinct from $u_G$, 
exactly $\frac{d_G(u)}{2}$ 
of the edges incident with $u$ have the same color,
and exactly $\frac{d_G(u_G)}{2}+1$ 
of the edges incident with $u_G$ have the same color.
\end{enumerate}
\end{theorem}
Using Vizing's bound \cite{vi} on the chromatic index leads to our second result.
\begin{theorem}\label{theorem0}
Every graph of minimum degree at least $2$ has a majority $4$-edge-coloring.
\end{theorem}
Clearly, a graph containing a vertex of degree $1$ 
does not have a majority edge-coloring,
which motivates the minimum degree condition in Theorem \ref{theorem0}.
Furthermore, 
since class 2 graphs of minimum degree at least $2$ and maximum degree $3$
have no majority $3$-edge-coloring,
the number of colors in Theorem \ref{theorem0} is best possible
under this minimum degree condition.
In fact, if a graph $G$ of minimum degree at least $2$ 
has an induced subgraph $H$ 
such that $H$ is a class 2 graph of maximum degree $3$,
and all vertices of $H$ have degree $2$ or $3$ in $G$,
then $G$ has no majority $3$-edge-coloring.
We conjecture that these are all graphs for which $4$ colors are needed.

Our third result supports this conjecture.

\begin{theorem}\label{theorem1}
Every graph of minimum degree at least $4$ has a majority $3$-edge-coloring.
\end{theorem}
Since a graph containing a vertex of odd degree at least $3$ 
does not have a majority $2$-edge-coloring,
the number of colors in Theorem \ref{theorem1} is best possible
under the minimum degree condition in that result.
In Section \ref{section2} we prove our results,
and in a conclusion we discuss a variation of majority edge-colorings. 

\section{Proofs}\label{section2}

Theorem \ref{theorem-1} is a consequence of {\it Euler's Theorem} \cite{hi}.

\begin{proof}[Proof of Theorem \ref{theorem-1}]
(i) Let the multigraph $G'$ arise from $G$ 
by adding the edges of a perfect matching $M$ 
on the possibly empty set of vertices of odd degree.
Clearly, the multigraph $G'$ is connected 
and every vertex has even degree in $G'$.
Let $e_0e_1\ldots e_{m-1}$ be an {\it Euler tour} of $G'$,
where, provided that $M$ is not empty, 
we may assume that $e_{m-1}\in M$.
Setting $c(e_i)=(i\,\, {\rm mod}\,\, 2)+1$ 
for every index $i$ such that $e_i$ belongs to $G$,
yields the desired $2$-edge-coloring of $G$.

\medskip

(ii) Let $e_0e_1\ldots e_{m-1}$ be an Euler tour of $G$
such that $e_0$ is incident with $u_G$.
Now, setting $c(e_i)=(i\,\, {\rm mod}\,\, 2)+1$ 
for every index $i$,
yields the desired $2$-edge-coloring of $G$.
\end{proof}
Theorem \ref{theorem0} is a consequence of {\it Vizing's Theorem} \cite{vi}.

\begin{proof}[Proof of Theorem \ref{theorem0}]
Let $G$ be a graph of minimum degree at least $2$.
If $u$ is a vertex of degree $d$, and 
$d=d_1+\cdots+d_k$ is a partition of $d$ into positive integers $d_i$,
then the graph $H$ {\it arises from $G$ by splitting $u$ into vertices of degrees $d_1,\ldots,d_k$} if there is a partition
$N_G(u)=N_1\cup \ldots \cup N_k$ of $N_G(u)$
with $|N_i|=d_i$ for $i\in [k]$,
$V(H)=(V(G)\setminus \{ u\})\cup \{ u_1,\ldots,u_k\}$
for $u_1,\ldots,u_k\not\in V(G)$, and 
$E(H)=E(G-u)\cup\bigcup_{i\in [k]}\{ u_iv:v\in N_i\}$.
See Figure \ref{fig1} for an illustration.

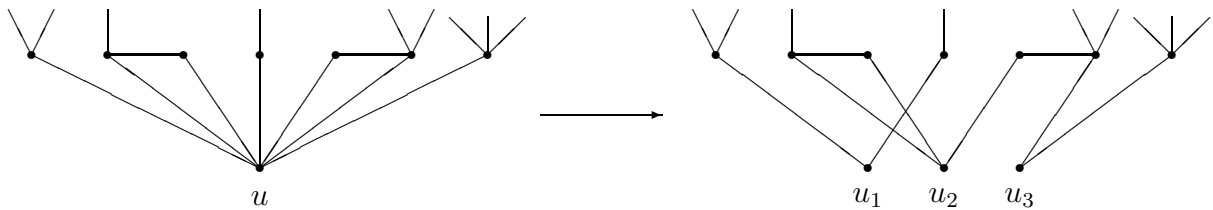
\begin{figure}[H]
\begin{center}
% This is a LaTeX picture output by TeXCAD.
% File name: [1.pic].
% Version of TeXCAD: 4.51
% Reference / build: 27-Nov-2018 (rev. a75)
% For new versions, check: http://texcad.sf.net/
% Options on the following lines.
%\grade{\on}
%\emlines{\off}
%\epic{\off}
%\beziermacro{\on}
%\reduce{\on}
%\snapping{\on}
%\pvinsert{% Your \input, \def, etc. here}
%\quality{8.000}
%\graddiff{0.005}
%\snapasp{1}
%\zoom{11.3137}
\unitlength 1mm % = 2.845pt
\linethickness{0.4pt}
\ifx\plotpoint\undefined\newsavebox{\plotpoint}\fi % GNUPLOT compatibility
\begin{picture}(158,25)(0,0)
\put(33,4){\circle*{1}}
\put(33,19){\circle*{1}}
\put(123,19){\circle*{1}}
\put(23,19){\circle*{1}}
\put(113,19){\circle*{1}}
\put(13,19){\circle*{1}}
\put(103,19){\circle*{1}}
\put(3,19){\circle*{1}}
\put(93,19){\circle*{1}}
\put(43,19){\circle*{1}}
\put(133,19){\circle*{1}}
\put(53,19){\circle*{1}}
\put(143,19){\circle*{1}}
\put(63,19){\circle*{1}}
\put(153,19){\circle*{1}}
\put(3,19){\line(2,-1){30}}
\put(33,4){\line(-4,3){20}}
\put(23,19){\line(2,-3){10}}
\put(33,4){\line(0,1){15}}
\put(43,19){\line(-2,-3){10}}
\put(33,4){\line(4,3){20}}
\put(63,19){\line(-2,-1){30}}
\put(13,19){\line(1,0){10}}
\put(103,19){\line(1,0){10}}
\put(43,19){\line(1,0){10}}
\put(133,19){\line(1,0){10}}
\put(3,19){\line(-1,2){3}}
\put(93,19){\line(-1,2){3}}
\put(3,19){\line(1,2){3}}
\put(93,19){\line(1,2){3}}
\put(13,19){\line(0,1){6}}
\put(103,19){\line(0,1){6}}
\put(33,19){\line(0,1){6}}
\put(123,19){\line(0,1){6}}
\put(53,19){\line(1,2){3}}
\put(143,19){\line(1,2){3}}
\put(53,19){\line(-1,2){3}}
\put(143,19){\line(-1,2){3}}
\put(63,19){\line(1,1){5}}
\put(153,19){\line(1,1){5}}
\put(63,19){\line(0,1){5}}
\put(153,19){\line(0,1){5}}
\put(63,19){\line(-1,1){5}}
\put(153,19){\line(-1,1){5}}
\put(33,0){\makebox(0,0)[cc]{$u$}}
\put(123,4){\circle*{1}}
\put(133,4){\circle*{1}}
\put(113,4){\circle*{1}}
\put(153,19){\line(-4,-3){20}}
\put(133,4){\line(2,3){10}}
\put(93,19){\line(4,-3){20}}
\put(113,4){\line(2,3){10}}
\put(103,19){\line(4,-3){20}}
\put(123,4){\line(-2,3){10}}
\put(133,19){\line(-2,-3){10}}
\put(113,0){\makebox(0,0)[cc]{$u_1$}}
\put(123,0){\makebox(0,0)[cc]{$u_2$}}
\put(133,0){\makebox(0,0)[cc]{$u_3$}}
\put(70,11){\vector(1,0){16}}
\end{picture}
\end{center}
\caption{Splitting a vertex $u$ of degree $7$ into vertices of degrees $2$, $2$, and $3$.}\label{fig1}
\end{figure}
Now, let $G^*$ arise from $G$ by splitting every vertex of degree $d>3$
into vertices of degrees 
\begin{itemize}
\item $3,\ldots,3$, if $d\equiv 0\,{\rm mod}\,3$,
\item $2,2,3,\ldots,3$, if $d\equiv 1\,{\rm mod}\,3$, and
\item $2,3,\ldots,3$, if $d\equiv 2\,{\rm mod}\,3$.
\end{itemize}
Note that there is a natural bijection between the edges of $G$ and those of $G^*$.
By Vizing's Theorem \cite{vi}, the graph $G^*$ has a proper $4$-edge-coloring,
which yields a majority $4$-edge-coloring of $G$.
In fact, we obtain an edge-coloring of $G$ such that,
for every vertex of degree $d$ at least $4$,
at most $(d+2)/3$ of the incident edges have the same color.
\end{proof}
We proceed to the proof of Theorem \ref{theorem1}.

\begin{proof}[Proof of Theorem \ref{theorem1}]
Let $G$ be a graph of minimum degree $\delta$ at least $4$.
Let $V(G)=D\cup A\cup C$ be the {\it Gallai-Edmonds decomposition} of $G$, 
that is, 
$D$ is the set of all vertices of $G$ that are missed by some maximum matching,
$A$ is the set of all vertices of $G$ outside of $D$ that have a neighbor of $D$,
and $C$ contains the remaining vertices, cf.~\cite{lopl}. 

Let $D'$ be the set of isolated vertices in $G[D]$.

\begin{claim}\label{claim1}
It is possible to select, for every vertex $u$ in $D'$, 
exactly one edge incident with $u$
in such a way that every vertex $v$ in $A$ is incident with at most 
$\left\lfloor\frac{d_G(v)}{2}\right\rfloor$ of the selected edges.
\end{claim}
\begin{proof}[Proof of Claim \ref{claim1}]
Let the network $N$ arise from the bipartite subgraph of $G$
with partite sets $D'$ and $A$ that contains all edges of $G$
incident with the vertices in $D'$ by 
\begin{itemize}
\item adding a source vertex $s$, and connecting $s$ to each vertex in $D'$ by an arc of capacity $1$,
\item adding a sink vertex $t$, and connecting each vertex $v$ in $A$ to $t$ by an arc of capacity $\left\lfloor\frac{d_G(v)}{2}\right\rfloor$,
\item orienting all edges of $G$ between $D'$ and $A$ towards $A$, and assigning capacity $1$ to these arcs.
\end{itemize}
By the {\it Max-Flow-Min-Cut Theorem} \cite{ff} and the {\it Integral Flow Theorem} \cite{df}, 
the claimed statement is equivalent to the statement that a maximum flow in $N$ has value $|D'|$.
Hence, in order to complete the proof, we suppose, for a contradiction,
that $f$ is an integral maximum flow in $N$ of value less than $|D'|$.
Let $X$ be the set of all vertices of $N$ 
that are reachable from $s$ on a directed $f$-augmenting path,
that is, the set $X$ defines an $s$-$t$-cut of minimum capacity,
for all arcs leaving $X$, the flow value is the capacity, and
for all arcs entering $X$, the flow value is $0$.

Let $X_D=X\cap D'$ and $X_A=X\cap A$.

Let $D_0\cup D_1\cup D_2$ be a partition of $X_D$ such that
\begin{itemize}
\item $D_0$ contains the vertices $u$ from $X_D$ with $f((s,u))=0$,
\item $D_1$ contains the vertices $u$ from $X_D$ with $f((s,u))=1$ 
and $f((u,v))=1$ for some vertex $v$ in $X_A$,
\item $D_2$ contains the vertices $u$ from $X_D$ with $f((s,u))=1$
and $f((u,v))=0$ for every vertex $v$ in $X_A$.
\end{itemize}
Since the value of $f$ is less than $|D'|$, we have $D_0\not=\emptyset$.
By the definition and properties of $X$, 
no vertex in $D_0\cup D_1$ has a neighbor in $A\setminus X_A$,
and $f((v,t))=\left\lfloor\frac{d_G(v)}{2}\right\rfloor$
for every vertex $v$ in $X_A$,
which implies 
\begin{eqnarray}\label{e1}
|D_1|&=&\sum\limits_{v\in X_A}\left\lfloor\frac{d_G(v)}{2}\right\rfloor.
\end{eqnarray}
For every vertex $u$ in $D_2$, 
the unique vertex $v$ from $A\setminus X_A$
with $f((u,v))=1$ is the only neighbor of $u$ in $A\setminus X_A$.
Therefore, the number $m$ of edges of $G$ between $X_D$ and $X_A$ is
at least $\delta(|D_0|+|D_1|)+(\delta-1)|D_2|$, and 
at most $\sum\limits_{v\in X_A}d_G(v)$, which implies
\begin{eqnarray}\label{e2}
\sum\limits_{v\in X_A}d_G(v) 
& \geq & \delta(|D_0|+|D_1|)+(\delta-1)|D_2|
\stackrel{(\ref{e1})}{\geq} \underbrace{\delta |D_0|}_{>0}
+\sum\limits_{v\in X_A}\delta\left\lfloor\frac{d_G(v)}{2}\right\rfloor
> \sum\limits_{v\in X_A}\delta\left\lfloor\frac{d_G(v)}{2}\right\rfloor.
\end{eqnarray}
For integers $\delta$ and $d$ with $3\leq \delta\leq d$, 
it is easy to verify that $\delta\left\lfloor\frac{d}{2}\right\rfloor\geq d$, 
which yields a contradiction to (\ref{e2}).
This completes the proof of Claim \ref{claim1}.
\end{proof}
The properties of the {\it Gallai-Edmonds decomposition} imply
that $G[C]$ has a perfect matching $M_C$,
that there is a matching $M_A$ using edges between $A$ and $D$
that connects each vertex from $A$ to a distinct component of $G[D]$, and
that every component of $G[D]$ is {\it factor-critical}.

We now construct a subset $E_1$ of the edge set $E(G)$ of $G$
as follows starting with the empty set:
\begin{itemize}
\item We add to $E_1$ all $|D'|$ selected edges as in Claim \ref{claim1}.
\item We add $M_C$ to $E_1$.
\item For every vertex $v$ from $A$ that is not incident with a selected edge,
we add to $E_1$ the unique edge from $M_A$ incident with $v$.
Let $M_A'$ be the subset of $M_A$ added to $E_1$.
\item For every component $K$ of $G[D]$ of order at least $3$ 
such that some vertex $x$ of $K$ is incident with an edge from $M_A'$,
we add to $E_1$ a perfect matching of $K-x$.
\item For every component $K$ of $G[D]$ of order at least $3$ 
such that no vertex of $K$ is incident with an edge from $M_A'$,
we add to $E_1$ a perfect matching of $K-x$ for some vertex $x$ of $K$
as well as one further edge of $K$ incident with $x$.
\end{itemize}
Up to some small modifications explained below, this completes the description of $E_1$.

By construction, 
the spanning subgraph $G_1$ of $G$ with edge set $E_1$ satisfies 
\begin{eqnarray}\label{e3}
1\leq d_{G_1}(u)\leq \left\lfloor\frac{d_G(u)}{2}\right\rfloor
\mbox{ for every vertex $u$ of $G$.}
\end{eqnarray}
Let $G_2$ be the spanning subgraph of $G$ with edge set $E(G)\setminus E_1$.

For every component $K$ of $G_2$ such that 
all vertices of $K$ have even degree in $G_2$, 
$K$ has an odd number of edges, and
all vertices from $V(K)$ have degree $1$ in $G_1$,
we select any edge $e_K$ from $K$ and move it from $G_2$ to $G_1$.
Note that $K-e_K$ contains exactly two vertices of odd degree, and, hence, is still connected.
Furthermore, since $G$ has minimum degree at least $4$, 
it follows that (\ref{e3}) still holds after each such modification.
Having performed these modifications for each such component of $G_2$,
every component $K$ of (the modified) $G_2$ now
\begin{itemize}
\item either contains at least one vertex of odd degree in $K$,
\item or all vertices of $K$ have even degrees in $K$, 
and the number of edges of $K$ is even,
\item or all vertices of $K$ have even degrees in $K$, 
the number of edges of $K$ is odd, and
$K$ contains a vertex $u_K$ such that the degree of $u_K$ in $G_1$ is at least $2$.
\end{itemize}
The components of $G_2$ as in the final point are called {\it type 2} components,
and the remaining components of $G_2$ are called {\it type 1} components.

We are now in a position to describe a majority $3$-edge-coloring $c:E(G)\to [3]$.
\begin{itemize}
\item For all edges $e$ of $G_1$, let $c(e)=3$.
\item For every component $K$ of $G_2$ that is of type $1$,
let $c:E(K)\to [2]$ be as in Theorem \ref{theorem-1}(i) (applied to $K$ as $G$).
\item For every component $K$ of $G_2$ that is of type $2$,
let $c:E(K)\to [2]$ be as in Theorem \ref{theorem-1}(ii) 
(applied to $K$ and $u_K$ as $G$ and $u_G$).
\end{itemize}
It is now easy to verify that $c$ is a majority $3$-edge-coloring of $G$,
which completes the proof.
\end{proof}

\section{Conclusion}

The most natural question motivated by our results is 
which graphs of minimum degree at least $2$ 
do not have a majority $3$-edge-coloring. 

As a variation of majority edge-colorings, 
we propose to study {\it $\alpha$-majority edge-colorings} for $\alpha\in (0,1)$,
where at most an $\alpha$-fraction of the edges 
incident with each vertex 
are allowed to have the same color.
If $k$ is a positive integer at least $2$, 
then every positive integer at least $k(k-1)$ 
can be written as a non-negative integral linear combination of $k$ and $k+1$.
Using this fact, a straightforward adaptation of the proof of Theorem \ref{theorem0}
yields the following statement: 
{\it If a graph $G$ has minimum degree at least $k(k-1)$,
then $G$ has a $\frac{1}{k}$-majority $(k+2)$-edge-coloring.}
A probabilistic argument implies that, for a sufficiently large minimum degree,
one color less suffices.

\begin{theorem}\label{theorem3}
For every integer $k$ at least $2$,
there is a positive integer $\delta_k$ such that 
every graph of minimum degree at least $\delta_k$
has a $\frac{1}{k}$-majority $(k+1)$-edge-coloring.
\end{theorem}
\begin{proof}
Let $G$ be a graph of minimum degree $\delta$ at least $\delta_k$,
where we specify $\delta_k$ later. 
Let $c:E(G)\to [k+1]$ be a random $(k+1)$-edge-coloring,
where we choose the color of each edge uniformly and independently at random.
For every vertex $u$ of $G$, we consider the {\it bad event} $A_u$
that more than $\frac{1}{k}d_G(u)$ of the edges incident with $u$ have the same color.
For $d=d_G(u)$, the {\it union bound} and the {\it Chernoff inequality}, cf.~\cite{more}, imply 
\begin{eqnarray*}
\mathbb{P}\left[A_u\right] 
& \leq & (k+1)\mathbb{P}\left[{\rm Bin}\left(d,\frac{1}{k+1}\right)>\frac{d}{k}\right]\hspace{3cm}\mbox{(union bound)}\\
&=& (k+1)\mathbb{P}\left[{\rm Bin}\left(d,\frac{1}{k+1}\right)>
\left(1+\frac{1}{k}\right)\frac{d}{k+1}\right]\\
&\leq &(k+1)e^{-\frac{d}{3k^2(k+1)}}.\hspace{5.7cm}\mbox{(Chernoff inequality)}
\end{eqnarray*}\\[-16mm]

In order to complete the proof, 
we use the {\it weighted Lov\'{a}sz Local Lemma}, cf.~\cite{more},
to show that with positive probability none of the bad events $A_u$ occurs.
Let $p=(k+1)e^{-\frac{\delta}{3k^2(k+1)}}$
and, for every vertex $u$ of $G$, let $t_u=\left\lfloor\frac{d_G(u)}{\delta}\right\rfloor$.
Note that $d_G(u)\geq \delta$ implies that $t_u$ is a positive integer,
and that $2t_u=2\left\lfloor\frac{d_G(u)}{\delta}\right\rfloor\geq \frac{d_G(u)}{\delta}$.
Choosing $\delta_\alpha$ sufficiently large, 
we may assume that $p\leq \frac{1}{4}$, and, hence,
$\mathbb{P}\left[A_u\right]\leq p^{\frac{d_G(u)}{\delta}}\leq p^{t_u}$.
Since, for every vertex $u$ of $G$, 
the event $A_u$ is determined only by the colors of the edges incident with $u$,
which are chosen uniformly and independently at random,
the event $A_u$ is mutually independent of all events $A_v$
except for those for which $v$ is a neighbor of $u$.
We obtain
\begin{eqnarray*}
\sum\limits_{v\in N_G(u)}(2p)^{t_v}
\leq 2 p d_G(u)
\leq 4p\delta t_u
=\underbrace{\left(4(k+1)e^{-\frac{\delta}{3k^2(k+1)}}\delta\right)}_{\to 0\,\,for\,\,\delta\to\infty}t_u,
\end{eqnarray*}\\[-6mm]
which is at most $t_u/2$ for $\delta_k$ sufficiently large.
Now, by the weighted Lov\'{a}sz Local Lemma,
the edge-coloring $c$ is a $\frac{1}{k}$-majority $(k+1)$-edge-coloring
with positive probability, which completes the proof.
\end{proof}
Our Theorem \ref{theorem1} implies that $4$ 
is the smallest possible value for $\delta_2$.

\pagebreak

\noindent {\bf Acknowledgement.}
The research reported in this paper was carried out at the 
{\it 25th C5 Graph Theory Workshop} 
organized by Ingo Schiermeyer in Rathen, May 2022.
After a pause due to the COVID-19 pandemic,
we very strongly felt the value of the creative, collaborative, and enjoyable atmosphere 
and the direct personal exchange at that workshop.
We express our gratitude to Ingo, 
not just for this year but also for the long tradition of this wonderful meeting.

\end{document}